\makeindex             

\documentclass[11pt]{article}
\usepackage{amsfonts}
 \usepackage{amssymb}
\usepackage{amsmath,amsxtra}

\usepackage[usenames]{color}

\usepackage[all,matrix,arrow,curve]{xy}

\usepackage{mathrsfs}

\DeclareFontFamily{OT1}{pzc}{}
\DeclareFontShape{OT1}{pzc}{m}{it}{<-> s * [1.15] pzcmi7t}{}
\DeclareMathAlphabet{\mathpzc}{OT1}{pzc}{m}{it}

\newtheorem{theorem}{Theorem}
\newtheorem{deff}{Definition}

\newtheorem{proposition}{Proposition}

\newtheorem{lemma}{Lemma}

\newtheorem{rem}{Remark}

\newcommand{\npa}{\addtocounter{num}{1} \noindent
{\bf \arabic{section}.\arabic{num}}\;\;}

\newcommand{\proof}{{\bf Proof.}~}

\newcommand{\mto}{\mapsto}
\newcommand{\bqa}{\begin{eqnarray}}
\newcommand\eqa {\end{eqnarray}}
\newcommand{\beq}{\begin{eqnarray}}
\newcommand{\beqn}{\begin{eqnarray}\nonumber}
\newcommand{\eeq}{\end{eqnarray}}
\newcommand{\be}{\begin{array}}
\newcommand{\ee}{\end{array}}

 \newcommand{\pr}{\partial}
 \newcommand{\pt}{\partial}
 \newcommand\na {\nabla}

   \newcommand\vf\varphi

 \newcommand{\Hom}{\mathrm{Hom}}
 \newcommand{\End}{\mathrm{End}}
 
 \newcommand{\Id}{\mathrm{Id}}

 \newcommand{\uHom}{\underline{\mathrm{Hom}}}

 \newcommand{\rk}{\mathrm{rk}}
 \newcommand{\md}{\mathrm{d}}

 \newcommand{\cV}{{\cal V}}

 \newcommand{\cA}{{\cal A}}
 
 \newcommand{\cM}{{\cal M}}
 
 \newcommand{\cO}{{\cal O}}
 
 \newcommand{\cT}{{\cal T}}

 \newcommand{\cJ}{\mathcal{J}}
 
 \newcommand{\cN}{{\cal N}}

 \newcommand{\R}{{\mathbb R}}
 \newcommand{\Z}{{\mathbb Z}}

\begin{document}
\author{G. Bonavolont\`a\thanks{G. Bonavolont\`a thanks the Luxembourgian National Research Fund for support via AFR grant PhD-09-072.}, A. Kotov}

\title{{\bf On the space of super maps between smooth super manifolds}}

  \def\sp{\mathfrak sp}
  \def\sll{\mathfrak sl}
  \def\g{{\mathfrak g}}
  \def\gl{{\mathfrak gl}}
  \def\su{{\mathfrak su}}
  \def\so{{\mathfrak so}}
  \def\sll{{\mathfrak sl}}
  \def\h{{\mathfrak h}}

  \def\P{{\mathbb P}}
  \def\H{\mathbb H}

   \def\a{\alpha}
   \def\b{\beta}
   \def\t{\theta}
   \def\la{\lambda}
   \def\e{\epsilon}
   \def\ga{\gamma}
   \def\de{\delta}
   \def\De{\Delta}
   \def\om{\omega}
   \def\Om{\Omega}

   \def\i{\imath}

 \def\gr{\g^{\scriptscriptstyle\mathrm{gr}}}
 \def\godd{\g_{\scriptscriptstyle 1}}
 \def\geven{\g_{\scriptscriptstyle 0}}
 \def\grodd{\gr_{\scriptscriptstyle 1}}
 \def\greven{\gr_{\scriptscriptstyle 0}}

  \def\sst{\scriptscriptstyle}

  \def\sot{{\;{\scriptstyle \otimes}\;}}
  \def\st{{\sst\times}}

  \def\df{{\sst\mathrm{def}}}

\def\scrE{\mathscr{E}}
  \def\scrF{\mathscr{F}}
  \def\scrJ{\mathscr{J}}
  \def\scrR{\mathscr{R}}

\date{}
\maketitle

\begin{abstract}
{Mapping spaces of supermanifolds are usually thought as exclusively in functorial terms (i.e. trough the Grothendieck functor of points). In this work we
provide a geometric description of such mapping spaces in terms of infinite-dimensional super-vector bundles.}

\end{abstract}
\vspace{5mm}


 \section{Introduction}
Let ${\bf SVect}_{\mathbb{R}}$ be the category of super-vector spaces over $\mathbb{R}$ with the standard structure of tensor category and braiding corresponding to the Koszul sign rule convention (see $\cite{DM99}$). The linear algebra in ${\bf SVect}_{\mathbb{R}}$ has two notions of linear operations: categorical hom, $\Hom_{{\bf SVect}_{\mathbb{R}}}(V,W)$ is an ordinary vector space for each couple of super-vector spaces $V,W\in {\bf SVect}_{\mathbb{R}} $ and it consists of all $\mathbb{R}$-linear operators preserving the parity; the inner hom, $\uHom(V,W)$, defined as the adjoint functor to the tensor bi-functor, which is a super-vector space consisting of all $\mathbb{R}$-linear operators.
The setting change drastically passing to the category of finite dimensional smooth supermanifold $\bf{SMan}$. In notation of $\cite{DM99}$ each element $\mathbf{M} \in {\bf SMan}$ will be represented by the corresponding locally ringed space $(M,\cO_{M})$. This category has a natural monoidal category structure, with Cartesian product of two supemanifolds $\mathbf{M}=(M,\cO_{M}),\mathbf{N}=(N,\cO_{N})$ given by the co-product of the corresponding locally ringed space (see appendix). Analogously to ${\bf SVect}_{\mathbb{R}}$, there are two notions of mapping space in ${\bf SMan}$: the categorical hom, i.e. morphisms of locally ringed spaces
, still denoted as $\Hom(\mathbf{M},\mathbf{N})$ and the inner hom, $\uHom(\mathbf{M},\mathbf{N})$ (see appendix, \cite{DM99}). However one treats such spaces only as functors, i.e. trough the Grothendieck functor of points, without spelling out the structure of infinite-dimensional supermanifolds.
 Many remarkable attempts to define a category of infinite dimensional supermanifolds have been performed: $\cite{A. Alldridge},\cite{Mol}$-$\cite{Sa-Wo}$.
 We intend to suggest a different approach to the same problem, that up to our opinion has the advantage of being less abstract and very close to the coordinate view point used in physics for ``super-fields''.
Our aim is to describe both mapping space as particular infinite dimensional super-vector bundle over the infinite dimensional smooth manifold mapping space $C^{\infty}(M,N)$, given by the set of all smooth mappings from $M$ to $N$ according to $\cite{Mic}, \cite{Mic1}$. In paragraph $\S \ref{first-paragraph}$ the starting point is the classical result known as Batchelor theorem, which affirms that each smooth supermanifold is non-canonically isomorphic to a vector bundle over a classical manifold with an odd fiber.
Then we shall identify any morphism of smooth supermanifolds with a certain bundle map and finally, by the use of an auxiliary connection, with a morphism of vector bundles. This proof consists of many steps and some preliminary work on classical jets of a smooth manifold $M$ is needed. In paragraph $\S \ref{second-paragraph}$ the framework change: we sketchy review the theory of infinite dimensional smooth manifold following P. Michor ($\cite{Mic}, \cite{Mic1}$). In fact it is necessary to recall some general results in the setting of global analysis in order to describe the infinite dimensional super vector bundle structure of $\Hom(\mathbf{M},\mathbf{N})$ and $\uHom(\mathbf{M},\mathbf{N})$.  Standing this result $\bf{SMan}$ will be identified with a full subcategory of a bigger category containing the aforementioned infinite dimensional super-vector bundle. Finally the link between our construction and the categorical definition of $\uHom(\mathbf{M},\mathbf{N})$ is discussed.

\section{Superfields vs. vector bundles}\label{first-paragraph}
\vskip 0.5cm
{\bf Notation}. {\em We denote with ${\bf Man}$ the category of finite dimensional smooth manifolds. Let $M,N\in {\bf Man}$ we denote the mapping space $C^{\infty}(M,N)$ with $\Hom_{{\bf Man}}(M,N)$} (or even neglecting ${\bf Man}$ when it is clear from the context)
\vskip0.5cm

Let $\mathbf{M}$ be a smooth supermanifold of superdimension $(p,q)$ over a base $M$. According to the fundamental result of smooth supergeometry \cite{Gawedzki}, there exists a vector bundle $V$ of rank $q$ over $M$ such that $\mathbf{M}$ is diffeomorphic as a supermanifold to $\Pi V$, that is, to the total space of $V$ with the reversed parity of fibres, the algebra of functions on which is canonically identified with $\Gamma (\Lambda^\bullet V^*)$. In other words, each supermanifold is diffeomorphic to a bundle over a classical manifold with an odd fiber. However, a general morphism of supermanifolds $f\colon \mathbf{M}\to \mathbf{N}$, which covers a base map $f_{\sst 0}\colon M\to N$, is not a bundle map.
Later on we shall identify any morphism of smooth supermanifolds with a certain bundle map and then, by the use of an auxiliary connection, with a morphism of vector bundles.
We intend to prove the following:
\begin{proposition}\label{space_of_fields}
Let $\mathbf{M}$ and $\mathbf{N}$ be supermanifolds which are super diffeomorphic to $\Pi W$ and $\Pi V$ for vector bundles $W\to M$ and $V\to N$. There is a one-to-one non canonical correspondence between morphisms $f\in \Hom(\mathbf{M},\mathbf{N})$, which cover a base map $f_{0}:M\to N$ and sections of even degree of the following super-bundle over $M$
 \beq\label{result1}
\xymatrix
{\widetilde{V}:=S_{+}^{(k)} ( W^*[1])\otimes f_0^* (TN)\oplus S^{(k)} ( W^*[1])\otimes f_0^*(V[1])\ar[d]\\
M\\
}
\eeq
with $k$ the rank of $W$, $S^{(k)}(W^*[1])$ ($S_{+}^{(k)}(W^*[1])$) the graded symmetric algebra up to order $k$ of $W^*[1]$ (resp. augmented), and $f_0^* (TN),f_0^* (V[1])$ the pull-back bundle of $TN,V[1]$ along $f_0$.
\end{proposition}
The graded morphisms, elements of $\uHom(\mathbf{M},\mathbf{N})$, are obtained by the last proposition dropping off the condition about the degree\footnote{We will moreover assume $V$ and $W$ $\Z-$graded.}and considering the projective limit of algebras.

\begin{proposition}\label{infinite vector bundle}The mapping space that to each  $f_{0}\in \Hom(M,N)$ associates the space of section of $\widetilde{V}$ as in proposition $(\ref{space_of_fields})$, is a locally trivial infinite dimensional vector bundle over the smooth manifold $\Hom(M,N)$ .
\end{proposition}

\subsection{Setup and proof}

\vskip 3mm\npa Let $M$ be a smooth manifold, $\cO_M$ be the sheaf of local smooth functions on $M$.
Let $J^k (M)$ be the bundle of $k-$jets of local smooth functions on $M$, $\cJ^k_M$ be the sheaf of its local sections, and
$j^k \colon \cO_M\to \cJ^k_M$ the morphism of $\R-$sheaves, which associates to any (local) smooth function its $k-$jet.
Given any coordinate chart $(U,\{x^i\})$,
we can identify $\cJ^k_M (U)$ with $C^\infty (U)\otimes  \R [\de x^1,\ldots ,\de x^n]/\mathpzc{m}^{k+1}$, where
$\de x^i$ are generators of the polynomial algebra associated to $\{x^i\}$ and
$\mathpzc{m}$ is an ideal of polynomials vanishing at $0$. Then for each smooth $h (x)$ on $U$ one has
  \beq
  \begin{split}
   j^k (h)(x,\de x)=h(x+\de x)\,& \mathrm{mod}\, \mathpzc{m}^{k+1} \\
   & \sim
   \sum\limits_{\mu_1+\ldots \mu_n \le k} \frac{(\de x^1)^{\mu_1}\ldots (\de x^n)^{\mu_n}}{\mu_1 !\ldots \mu_n!}
   (\pr_{x^1})^{\mu_1}\ldots (\pr_{x^n})^{\mu_n} f (x)\,,
   \end{split}
  \eeq
  where $\mu =(\mu_1,\ldots ,\mu_n)$ is a multi-index with non-negative integer components.
For simplicity, further we shall use the following compact form for multi-index notations:
  \beq
   \sum\limits_{\mu_1+\ldots \mu_n \le k} \frac{(\de x^1)^{\mu_1}\ldots (\de x^n)^{\mu_n}}{\mu_1 !\ldots \mu_n!}
   (\pr_{x^1})^{\mu_1}\ldots (\pr_{x^n})^{\mu_n} f (x)=
   \sum\limits_{\mid \mu \mid =0}^k \frac{(\de x)^\mu}{\mu !} \pr_x^\mu f (x)\,.
  \eeq
Let $\{y^i\}$ be another coordinate system,
such that $y^i =f^i (x)$, then the corresponding change of polynomial generators of $\cJ^k_M (U)$ over $C^\infty (U)$ is
  \beq\label{formal_change}
   \de y^i & = &  \left( f^i (x +\de x)-f^i (x) \right)\,\mathrm{mod}\, \mathpzc{m}^{k+1} \sim
   \sum\limits_{\mid \mu \mid =1}^k \frac{(\de x)^\mu}{\mu !} \pr_x^\mu f^i (x) \,.
  \eeq
Besides the canonical multiplication on local functions, given by the embedding $i^k\colon f\mto f\otimes 1$
for each $f$ in $\cO_M$, the sheaf
$\cJ^k_M$ admits the second structure of an $\cO_M-$algebra, determined by the embedding $j^k$.

\vskip 3mm\npa These two structures (of ringed spaces over $M$) have
 the following clear geometrical meaning (cf.~\cite{Grothendieck65}, also \cite{Spencer69}). Let $M\times M$ be the Cartesian product of two copies of $M$ with natural projections $\sigma$ and $\tau$ onto the factors:

  \beq
    \xymatrix{ & M\times M \ar[dl]_{\sigma}\ar[dr]^{\tau} & \\
               M && M  }
  \eeq

\noindent Let $M\hookrightarrow M\times M$ be the diagonal embedding and let
$\mathpzc{m}_M$ be the sheaf of
local smooth functions on $M\times M$ vanishing on the diagonal.
The coset sheaf $\cO_{M\times M}/\mathpzc{m}_M^{k+1}$ can be thought of as the structure
sheaf of the $k-$th order neighborhood of the diagonal, denoted as $M^{\sst (k)}$.
The restrictions of the projections $\sigma$ and $\tau$ to $M^{\sst (k)}$, which we denote as $\sigma_k$ and $\tau_k$, respectively, give us the structure of two ringed spaces mentioned above, such that $i^k=\sigma_k^*$ and $j^k=\tau_k^*$. Further we shall treat $\cO_{M\times M}$ and its coset sheaves as sheaves of $\cO_M-$algebras with respect to the first module structure, unless another assumption is explicitly stated.
Then $\cJ^k_M$ is isomorphic to $\cO_{M\times M}/\mathpzc{m}_M^{k+1}$.

\vskip 3mm\noindent In fact, in the definition of the $k-$th order neighborhood
one can replace $\cO_{M\times M}$ with $\cO_{\tilde U}$, where $\tilde{U}$ is any tubular neighborhood of the diagonal.
More precisely, let $r (x,x')$ be a smooth function on $U\times U$ which represents an element $[r]$ in the quotient sheaf of algebras (here $\{x'\}$ are the same coordinates on the copy of $U$). Then the image of $[r]$ in $\cJ^k (U)$ is
  \beq
   r(x,x+\de x)\,\mathrm{mod}\, \mathpzc{m}_X^{k+1}\sim
   \sum\limits_{\mid \mu \mid =0}^k \frac{(\de x)^\mu}{\mu !} \pr_{x'}^\mu r (x,x')_{\mid x=x'}\,.
  \eeq

\vskip 2mm\npa
Given an arbitrary (torsion free) connection $\na$ in $TM$, we can establish a diffeomorphism between
  a neighborhood of the zero section in $TM$ and some $\tilde{U}$.\footnote{We use the fact that the normal bundle to the diagonal in $M\times M$ is
  naturally isomorphic to $TM$.} Indeed, let $z(x,\xi,t)$ be the solution of the Cauchy problem for geodesic equation with the initial data $(x,\xi)\in TM$,
   then the required map is $(x,\xi)\mto (x,z(x,\xi,1))$.
The image of the zero section in $TM$ under the map above coincides with the diagonal.
As a simple corollary, the $k-$th order neighborhood of the diagonal is isomorphic to the $k-$th order neighborhood of the zero section in $TM$, which we denote as $TM_0^{\sst (k)}$.
 More precisely, let $\{x^i,\xi^i\}$ and $\{x^i,\de x^i\}$ be the local coordinates of $TM_0^{\sst (k)}$ and $M^{\sst (k)}$, respectively, which are canonically associated to
coordinates $\{x^i\}$ of $M$, and let
$\Gamma_{jl}^i (x)$ be the Christoffel symbols of $\na$ in these coordinates. Then,
taking into account the property $z^i (x,t\xi,1)=z^i (x,\xi,t)$, we obtain the following asymptotic expansion
of $z^i (x,\xi,1)$:
  \beq\label{expansion1}
  \begin{split}
   z^i (x,\xi,1)&=\sum\limits_{m=0}^k \frac{1}{m!}\pr_t^m z^i (x,t\xi,1)_{\mid t=0} \!\!\! \mod \xi^{k+1}\\
   &=\sum\limits_{m=0}^k \frac{1}{m!}\pr_t^m z^i (x,\xi,t)_{\mid t=0}  \!\!\!\mod \xi^{k+1}\,
  \end{split}
  \eeq
  where by $\mathrm{mod}\, \xi^{k+1}$ we mean the quotient modulo all terms of degree
  higher than $k$.
  Now we apply the geodesic flow equation for $z^i (x,\xi,t)$
  \beq\label{geodesic-eqn}
   \pr_t^2 z^i  =-\Gamma_{jl}^i(z)\,\pr_t z^j \pr_t z^l
  \eeq
with the initial data
  \beq\label{initial-data}
   z^i (x,\xi,0)=x^i\,, \hspace{3mm}\pr_t z^i (x,\xi,0)=\xi^i
  \eeq
to establish the required isomorphism of $k-$th order neighborhoods $(x,\xi)\mto (x,\de x (x,\xi))$:
  \beq
   \de x^i &=&\left(z^i (x,\xi,1)- x^i \right)\!\mod \xi^{k+1}= \\ \nonumber &=&
   \xi^i -  \frac{1}{2}\Gamma_{jl}^i (x)\,\xi^j\xi^l +  \frac{1}{6} \left(-\pr_{x^s} \Gamma_{jl}^i (x)+
    2\Gamma_{ps}^i (x)\Gamma_{jl}^p (x)\right)\xi^s\xi^j\xi^l +\ldots \,.
  \eeq

\vskip 2mm\noindent The structure sheaf of $TM_0^{\sst (k)}$ is canonically
isomorphic to $S^{(k)} (\cT^*_M)$ where $\cT^*_M$ is the sheaf of smooth sections of the cotangent bundle and
$S^{(k)}(\cT^*_M)$ is the quotient of the symmetric algebra of $\cT^*_M$ generated over $\cO_M$ by the ideal of elements of degree higher than $k$. Therefore $\cJ^k_M\simeq S^{(k)} (\cT^*_M)$ as sheaves of $\cO_M-$algebras.
  The next proposition provides us with an explicit formula for the isomorphism between $\cJ_M^k$ and
$S^{(k)} (\cT^*_M)$ (fixed by a torsion free connection $\na$ in $TM$) which we denote by $\Phi^{k}$.
Let $\chi^{k}$ be a derivation
of the quotient sheaf $S^{(k)} (\cT^*_M)$ determined by the symmetrized covariant derivative
 \beq
  S^{*} (\cT^*_M)\xrightarrow{\na}
  \cT^*\otimes S^{*} (\cT^*_M)\to S^{*+1} (\cT^*_M)\,.
 \eeq
 Here the last arrow is the symmetrization map\footnote{The symmetrized covariant derivative preserves the ideal of symmetric tensors of degree $>k$, thus $\chi^{k}$ is well-defined.}.

\begin{proposition}\label{explicit-isomorphism1}
For any local function $h$ the following identity holds: $\Phi^{k}\circ j^k (h)=\exp(\chi^{k})(h)$.
\end{proposition}
\proof Given that the statement, we are going to prove, is coordinately independent, it is enough to check it in arbitrary local coordinates $\{x^s\}$ for $h=x^i$.  In the associated coordinates $\{x^i, \xi^i\}$ on $TM$ the derivation
   $\chi^{k}$ will have the following expression:
   \beq\label{derivation-chi}
  \chi = \xi^s \frac{\pt}{\pt x^s}- \Gamma_{jl}^s (x)\, \xi^j \xi^l\frac{\pt}{\pt \xi^s} \mod \xi^{k+1}\,,
  \eeq
  where $\Gamma^i_{jk}$ are the Levi-Civita coefficients of $\na$. Let $\{t,z^s, p^s, p^s_2, p^s_3, \ldots\}$
be the local coordinates on the space of jets of parameterized curves in $M$ associated to $\{x^s\}$, where
$z^s$ is a copy of $x^s$ (used for the convenience),
$p^s_k$ corresponds to the $k-$th derivative of $z^s$,
 and $p^s=p^s_1$. The 2d order geodesic equation (\ref{geodesic-eqn}) will be written as follows:
 \beq\label{geodesic-eqn-in-jets}
  p^s_2  =-\Gamma_{jl}^s(z)\, p^j p^l\,.
 \eeq
Apparently, $p^s_k=D_t^k (z^s)$, where $D_t$ is the total derivative with respect to $t$.
Let us consider the infinite prolongation of (\ref{geodesic-eqn-in-jets}) as a subspace in the jet space.
 The total derivative is tangent to this subspace. Then $p_k$ becomes a function of $z$ and $p$ for all
$k\ge 2$, which allows to express $D_t$ in the form
 \beq\label{total-derivative}
D_t =\pt_t + p^s \frac{\pt}{\pt z^s}- \Gamma_{jl}^s(z)\, p^j p^l\frac{\pt}{\pt p^s}\,.
 \eeq
Using (\ref{expansion1}), (\ref{initial-data}) and (\ref{total-derivative}) and taking into account
(\ref{derivation-chi}), we obtain that
 \beq
  \Phi^{k} (x^i)=\sum\limits_{m=0}^k \frac{1}{m!}\left( \xi^s \frac{\pt}{\pt x^s}- \Gamma_{jl}^i(x)\, \xi^j \xi^l\frac{\pt}{\pt \xi^s} \right)^m (x^i)
   \mod \xi^{k+1}=
  \exp(\chi^{k}) (x^i)\,.
 \eeq
$\square$


\vskip 3mm\npa The last result can be generalized to the sheaf of $k-$jets of sections $\cJ^k (\cA)$ of an arbitrary sheaf of graded commutative algebras $\cA$ over $\cO_M$ which is freely generated by some finitely graded vector bundle $V=\oplus_{j\in\Z} V^j$ over $M$, i.e. $\cA=S^* (\cV^*)$,
where $\cV$ is the sheaf of smooth sections of $V$.

\begin{proposition}\label{k-identification}
The algebras $\cJ^k (\cA)$ and $S^{(k)} (\cT^*_M)\otimes_{\cO_M} \cA$ are isomorphic.
The choice of an isomorphism is uniquely fixed by the choice of a torsion free connection in $TM$ and a graded connection in $V$.
\end{proposition}
\proof One can check that $\cJ^k(\cA)$ is canonically isomorphic to $\tau_k^* (\cA)$ such that
$j^k (s)=\tau_k^* (s)$ for any local section $s$ of $\cA$, where $\tau_k$ is the second projection of $M^{\sst (k)}$
onto $M$. Taking into account that $\cA$ is freely generated as a graded commutative algebra over $\cO_M$ by
$\cV^*$ and thus $\cJ^k(\cA)$ is freely generated over $\cJ_k$ by $\tau_k^* (\cV^*)$, it is sufficient to prove that $\sigma_k^* \cV^*\simeq \tau_k^* \cV^*$ as $\cJ^k$-modules. The last statement is a corollary of the following fact. Let $\tilde U$ be a tubular (geodesic) neighborhood of the diagonal given by an arbitrary torsion free connection in $TM$ and let $\sigma$ and $\tau$ be the restrictions to $\tilde U$ of the corresponding projections onto $M$. Let us fix a connection in $V$. Then $\sigma^* V$ and $\tau^* V$ are canonically
isomorphic over $\tilde U$. Indeed, $\tilde U$ is isomorphic to $TM$ as a bundle over $M$ where the first bundle structure
is determined by $\sigma$. In this setting the $\sigma-$fibers are not only contractible, but moreover the contraction is uniquely fixed. This allows to trivialize $\tau^* V$ along $\sigma-$fibers in the unique way by use of parallel transport. Finally we apply the same technique as above to pass from the tabular neighborhood to the $k-$th order neighborhood of the diagonal. $\square$

\begin{rem} This procedure is actually dual to the ``quantization" of symbols of differential operators by use of connections.
\end{rem}

\noindent Let us denote by $\Phi^{k}$ the isomorphism between $\cJ^k (\cA)$ and
$S^{(k)} (\cT^*_M)\otimes \cA$ fixed by a torsion free connection $\na^{\sst TM}$ in $TM$
and a vector connection
$\na^{\sst V}$ in $V$. Let $\chi^{k}$ be the mapping determined by the symmetrized covariant derivative
$S^{*} (\cT^*_M)\otimes \cA \to S^{*+1} (\cT^*_M)\otimes \cA$ on the quotient $S^{(k)} (\cT^*_M)\otimes\cA$.

\begin{proposition}\label{explicit-isomorphism2}
\mbox

 \begin{itemize}
 \item The composition $\Phi^{k}\circ j^k \colon \cA\to S^{(k)} (\cT^*_M)\otimes \cA$ coincides with the restriction of $\exp(\chi^{k})$.
 \item Let $\Phi^{k}_1$ be the isomorphism determined by another couple of connection. Then there exists a nilpotent automorphism $\Psi$ of the algebra
 $S^{(k)}(\cT^*_M)\otimes S^*(\cA)$ over $\cO_M$, such that $\Phi^{k}_1=\Psi\circ \Phi^{k}$.
\end{itemize}
\end{proposition}
\proof The first part of Proposition \ref{explicit-isomorphism2} is a simple generalization of Proposition \ref{explicit-isomorphism1}. Let $\Phi^{k}_t$,
$t\in [0,1]$ be the family of isomorphisms determined by the linear interpolation between the new and old connections,
such that $\Phi^{k}_{t= 1} = \Phi^{k}_1$, $\Phi^{k}_{t=0} = \Phi^{k}$. The variation
 $\de^k =\pt_t \chi^{k}_t$ can be identified with a section of $\g^1$ where
 \beqn
  \g^r = S^{r+1} (\cT^*_M)\otimes \cT_M\oplus S^r (\cT^*_M)\otimes \End (\cA) \,.
 \eeq
 Apparently, $\g^+ =\oplus_{r>1} \g^r$ can be viewed as the sheaf of nilpotent derivations of $S^{(k)}(\cT^*_M)\otimes S^*(\cA)$ over $\cO_M$, thus
 $\de^k$ is a nilpotent derivation. Then
 \beqn
  \Phi^{k}_t\circ j^k = \exp(\chi^{k}_t)
  = \exp(\chi^{k}_t) \circ \exp(-\chi^{k}) \circ \Phi^{k}\circ j^k \,.
 \eeq
Taking into account that $\cJ^k (\cA)$ is generated by $j^k (\cA)$ over $\cO_M$, we immediately obtain that
 \beqn
  \Phi^{k}_t = \left(\exp(\chi^{k}_t) \circ \exp(-\chi^{k})\right) \circ\Phi^{k}\,,
 \eeq
which implies that $\pt_t \Phi^{k}_t =\psi_t \circ\Psi^{k}_{t}$, where
 \beqn
  \psi_t = \int\limits_0^1 \exp(\tau\chi^{k}_t)\circ \pt_t \chi^{k}_t \circ\exp(-\tau\chi^{k}_t)\md\tau =
   \int\limits_0^1 \exp\left( \tau ad_{\chi^{k}_t} \right) (\de^k)\md\tau\,.
 \eeq
It is clear that $ad_{\chi^{k}_t} (\g^r)\subset \g^{r+1}$, therefore $\psi_t\in\g^+$ and
hence $\Phi^{k}_1=\Psi_{t=1}\circ \Phi^{k}$, where $\Psi_t$ is the solution to
the Cauchy problem
\beqn
 \pt_t \Psi_t = \psi_t \circ \Psi_t \, , \hspace{3mm} \Psi_{t =0}  = \Id \,.
\eeq
 Given that $\psi_t$ belongs to the Lie algebra of nilpotent derivations
 of $S^{(k)}(\cT^*_M)\otimes S^*(\cA)$ over $\cO_M$, we deduce that $\Psi_t$ must be a nilpotent automorphism
 of the same algebra. Setting $\Psi=\Psi_{t=1}$ we finalize the proof.
$\square$

\vskip 3mm\npa Let $\mathbf{M}$ and $\mathbf{N}$ be supermanifolds which are super diffeomorphic to $\Pi W$ and $\Pi V$ for vector bundles $W\to M$ and $V\to N$, respectively.

\begin{proposition}\label{bundle-identification} There exists a one-to-one correspondence between morphisms of supermanifolds
$f\colon \mathbf{M}\to \mathbf{N}$, which cover a base map $f_{\sst 0}\colon M\to N$, and super bundle morphisms
$\tilde f\colon \Pi W\to \Pi(\tau_k^*V)$ such that the following diagram becomes commutative

  \beq
    \xymatrix{ & \Pi V \ar[r]& N\\
    \Pi W \ar[ur]^f\ar[r]_{\tilde f}\ar[d] & \tau_k^*\!(\Pi V) \ar[d]^{\sigma_k}\ar[u]_{\gamma_k} \ar[r] & N^{(k)}\ar[u]_{\tau_k} \\
               M \ar[r]_{f_0} & N&&  }
  \eeq

\noindent Here $\tau_k^*\!(\Pi V)$ is the pull-back of $\Pi V$ on $N^{\sst (k)}$
and $\gamma_k$ is the natural projection $\tau_k^*\!(\Pi V)\to \Pi V$ over $\tau_k\colon N^{\sst (k)}\to N$.
\end{proposition}
\proof In one direction, we lift $f_0$ to a morphism $f'\colon \mathbf{M}\to N$ (as zero on $W-$fibers).
Let us consider the super morphism $f'\times f\colon \mathbf{M}\to N\times \mathbf{N}$. By definition, this super morphism fits to the following commutative diagram:

  \beq
    \xymatrix{ & \Pi V & \\
    \Pi W \ar[ur]^f\ar[r]_{f'\times f}\ar[d] & N\times\Pi V \ar[d]_{\sigma}\ar[u]_{\gamma}  \\
               M \ar[r]_{f_0} & N&&  }
  \eeq

\noindent It is clear that $N\times\Pi V$ is diffeomorphic to $\Pi(\tau^*V)$, where $\tau$ is the projection
of $N\times N$ onto the second factor (as before), $\sigma$ and $\gamma$ are the projections of $N\times\Pi V$ onto
$N$ and $\Pi V$, respectively. The proof of Proposition \ref{bundle-identification} follows from the next Lemma.

\begin{lemma} The image of $f'\times f$ belongs to $\Pi(\tau_k^*V)$, where $k=\rk W$.
\end{lemma}
\proof The proof of the Lemma is a straightforward generalization of the next local observation.
Let us assume for simplicity that $V=0$.
Then $\Pi(\tau_k^*V)$ is just the $k$th order neighborhood of the diagonal $M^{\sst (k)}$ the structure sheaf of which over an open $U\subset M$ is the quotient $\cO_{U\times U}$ by the ideals of functions on $U\times U$ vanishing on the diagonal up to the $k-$th order.  We have to show that $(f'\times f)^* (F)$ for any such $F$ is zero. Indeed, if $f$ is given by
 \beq
  x^i =f_0^i (y) + \frac{1}{2}\sum\limits_{a,b}f^i_{ab}(y)\theta^a\theta^b +\ldots\,
 \eeq
where $y$ and $\theta$ are local even and odd coordinates of $\Pi W$, respectively, then
 \beq
  (f'\times f)^* (F) (y,\theta)=F(f(y),f(y))+\frac{1}{2}\sum\limits_j\pr_{z^j} F(f(y),z)_{\mid z=f(y)}
  f^j_{ab}(y)\theta^a\theta^b +\ldots\,.
 \eeq
On the other hand $(\theta)^\mu=0$ for each multi-index of length greater than $k$.
Therefore, if $F$ vanishes on the diagonal up to the $k-$th order
its image is identically zero. $\square$

\vskip 2mm \noindent In the opposite direction, the statement is tautologically trivial. $\square$

\vskip 3mm \npa As a corollary of Proposition \ref{bundle-identification} and using Proposition \ref{k-identification} we conclude that morphisms of supermanifolds covering $f_{0}$ are in one to one correspondence with section of $(\ref{result1})$ of even degree.

Thus ``superfields", elements of $\uHom(\mathbf{M},\mathbf{N})$, are in one-to-one correspondence with all sections of (\ref{space_of_fields}) where we consider the projective limit over $k$
-i.e. the fields become graded.\footnote{Generally, $V$ and $W$ should be also $\Z-$graded.}

\section{Global aspects}\label{second-paragraph}

 We sketchy review the theory of infinite-dimensional manifold structure for the mapping space $\Hom_{{\bf Man}}(M,N)$ following  $\cite{Mic},\cite{Mic1}$. In particular we review the $\mathfrak{D}$-topology and its refinement $\mathfrak{D}^{\infty}$-topology, and use it to recall that $\Hom(M,N)$ is a smooth manifold modelled on spaces $\mathfrak{D}(F)$ of smooth sections with compact support of vector bundles $F$ over $M$. A particular type of infinite-dimensional vector bundle over $\Hom(M,N)$ is introduced. Then the proof of Proposition $\ref{infinite vector bundle}$ follows. The last proposition describes the link between the categorical definition of $\uHom(\mathbf{M},\mathbf{N})$ and our construction.

\vskip1cm
\npa {\em Topology}\label{topology}. For each integer $n\in \mathbb{N}$, let $J^{n}(M,N)$ be the smooth fibre bundle of $n$-jets of maps from $M$ to $N$. For each $f\in \Hom(M,N) $ let  $j_{n}f:M\to J^{n}(M,N)$ being a smooth section, where $j_{n}f(x)$ is the-$n$-jet of $f$ at $x\in M$.
\begin{deff}(\cite{Mic}) Let $K=(K_{n}),\;n\in \mathbb{N}$ be a fixed sequence of compact subset of $M$ such that
$$K_{0}=\emptyset,\;K_{n-1}\subset K_{n}^{\circ},\; \forall n\;\;\;X=\bigcup_{n}K_{n}$$
($K_{n}^{\circ}$ open interior of $K_{n}$).
Then consider sequences $m=(m_{n}), U=(U_{n})$ for $n=0,1,...$ such that $m_{n}$ is non-negative integer and $U_{n}$ is open in $J^{m_{n}}(M,N)$. For each such pair $(m,U)$ of sequences define a set ${{O}}(m,U)\subset \Hom(M,N)$
$${O}(m,U)=\{f\in \Hom(M,N)|j^{m_{n}}f(K_{n}^{\circ})\subset U_{n}\;\;\;\forall n \in \mathbb{N} .\}$$
The $\mathfrak{D}$-topology on $\Hom(M,N)$ is given by taking all sets ${O}(m,U)$ as a basis for its open sets.
\end{deff}
The $\mathfrak{D}$-topology is {\em finer} than the Whitney-$C^{\infty}$ topology ($W^{\infty}$-topology). The $W^{\infty}$-topology is defined by $$W^{\infty}=\bigcup_{k=0}^{\infty}W^{k},$$ where each $W^{k}$ has the following properties:

1) A basis for open sets is given by all sets of the form $W^{k}(U)=\{g\in C^{k}(M,N), j^{k}g(M)\subseteq U \}$, where $U$ is open in $J^{k}(M,N)$.

2) If $d_{k}$ is a metric on $J^{k}(M,N)$\footnote{completely metrizable space} ($0 \leq k \leq \infty$) generating the topology, and if $f\in C^{k}(M,N)$, then the following is a neighborhood basis for $f$ in the $W^{k}$-topology:
$$N(f,k,\epsilon):=\{g\in C^{r}(M,N):\;\;d_{k}(j^{k}g(x),j^{k}f(x))<\epsilon(x),\;\;\forall x\in M\}$$
where $\epsilon \in C(M,]0,\infty[)$.

3)A sequence $f_n$ in $C^{r}(M,N)$ converges to $f\in C^{r}(M,N)$ in $W^{k}$-topology iff there exists a compact set $K\subseteq M$ such that $f_{n}$ equals $f$ off $K$ for all but finitely many $n$'s and $j^{k}f_{n}\to j^{k}f$ uniformly on $K$.

At point $2)$ by the use of the metric $d_k$, the concept of ``$k$-uniform convergence" has been introduced.
Observe that if $f_{n}\to f$ in the $\mathfrak{D}$-topology then $f_{n}\to f$ in the $W$-topology, in particular:

\begin{lemma}\cite{Mic} A sequence $(f_{n})\in \Hom(M,N)$ converges in the $\mathfrak{D}$-topology to $f$ iff there exists a compact set $K\subset M$ such that all but a finitely many of the $f_{n}$'s equal $f$ off $K$ and $j^{l}f_{n}\to j^{l}f$ uniformly on $K$, for all $l\in \mathbb{N}$.
\end{lemma}
\begin{deff}\label{equivalence}Let $M,N$ as above, if $f,g\in \Hom(M,N)$ and the set
$$\{x\in M\;\; f(x)\neq g(x)\}$$
is relatively compact in $M$, we call $f$ equivalent to $g$, $f\sim g$.
\end{deff}
This is clearly an equivalence relation. The $\mathfrak{D}^{\infty}$-topology on the set $\Hom(M,N)$ is now the weakest among all topologies on $\Hom(M,N)$ which are finer than $\mathfrak{D}$-topology and for which all equivalence classes of the above relation are open.

\vskip1cm
\npa {\em local model}. For each $f\in \Hom(M,N)$ define $\Gamma_{f}\subset M \times N $ the graph of $f$, and $\pi_{f}:E(f)\to \Gamma_{f}$
the vector bundle over $\Gamma_{f}$, given by (essentially pullback of the vector bundle $TN$ along $f$)
$$E(f)=\cup_{M} T_{(x,f(x))}(\{x\}\times N)\subset T_{\Gamma_{f}}(M\times N)$$
(with $T_{(x,f(x))}(\{x\}\times N)=\{0\}_{x}\times T_{f(x)}(N)$ and $\{0\}_{x}$ the zero section of $TM$ over $x$).
$E(f)$ is a realization of the normal bundle to $T_{\Gamma{f}}(\Gamma_{f})$ in $T_{\Gamma{f}}(M\times N)$.
Now let
$$exp^{N}:V\subset T(N)\to N $$
be exponential map associated to a torsion free connection, defined on a neighborhood $V$ of the zero section in $TN$. The exponential map gives a diffeomorphism of
$E(f)$ onto an open neighborhood $Z_{f}$ of $\Gamma_{f}$ in $M\times N$, given by
$$(0,v_{x})\in E(f)\to (x,exp_{f(x)}^{N}(v_{x}))\in Z_{f}\subset M\times N,$$
if $v_{x}\in V$.
It is clear that $\rho_f :Z_{f}\to \Gamma_{f}$ is a vector bundle with vertical projection\footnote{actually it is a {\em tubular neighborhood} of $\Gamma_{f}$ in $M\times N$}, i.e. $\rho_f(x,y)=(x,f(x))\;\;\forall (x,y)\in Z_{f}$. We will denote with $C^{\infty}(Z_{f})$ the set of sections of $(Z_{f},\rho_f,\Gamma_{f})$.
The exponential map gives a fibre preserving diffeomorphism $\tau_f:E(f)\to Z_f$

\beq
    \xymatrix{ E(f)\ar@{->}^-{\tau_f}[r]\ar@{->}_-{\pi_f}[dr]&\ar@{->}^-{\rho_f}[d]Z_f \\
      & \Gamma_f .\\ }
  \eeq

For $f\in \Hom(M,N)$ we define the set $U_{f}$ given by
$$U_{f}=\{h\in \Hom(M,N)| \Gamma_{h}\subset Z_{f},\;\;h \sim f \}.$$
In particular {\em each $U_{f}$ is open in the $\mathfrak{D}^{\infty}$-topology.}

In words: $h$ is in the open neighborhood of $f$ iff the corresponding graph $\Gamma_{h}$ is contained in $Z_{f}$ and $f\sim h$ according to definition $\ref{equivalence}$.

 The space of sections of the vector bundle $\pi_f:(E_{f})\to \Gamma_f$, with compact support endowed with $\mathfrak{D}^{\infty}$-topology, is a locally convex topological linear space (\cite{Mic} pag. $57$) .
It will be denoted as $\mathfrak{D}(E_{f})$.
\vskip1cm
\npa{\em smoooth structure}\label{smooth-structure}. We declare $U_{f}$ being a local chart for $f$ with the coordinate mapping $\phi_{f}:U_{f}\to \mathfrak{D}(E_{f})$
$$\phi_{f}(g)=\tau_f^{-1}(x,g(x))\;\;g\in U_{f},\;\;(x,f(x))\in X_{f}$$
and
$\psi_{f}:\mathfrak{D}(E_{f})\to U_{f}$ defined by
$$\psi_{f}(s)(x)=\pi_{N}\circ\tau_{f}\circ s(x,f(x)),\;\;s\in (\mathfrak{D}(E_{f})),x\in M,$$
with $\pi_{N}:M\times N\to N$ the canonical projection.
The maps $\psi_f$ and $\phi_f$ are inverse to each other and the chart change for $s\in \phi_{g}(U_g\cap U_f)$ is defined by
$$\phi_f\circ \psi_g(s)(x,f(x))=\tau_f^{-1}\circ \tau_g\circ s(x,g(x)).$$

\begin{theorem}\cite{Mic} $\Hom(M,N)$ is a smooth manifold with smooth atlas $(U_f,\psi_f)_{f\in \Hom(M,N)}$.
\end{theorem}

Another useful theorem is the following
\begin{theorem}\cite{Mic2}\label{Mic2} The infinite dimensional smooth vector bundle $T\Hom(M,N)$ (``tangent bundle of $\Hom(M,N)$") and the space of sections with compact support of the pullback vector bundle $C_c^{\infty}(M,f^{\ast} TN),\;\;\forall f\in \Hom(M,N),$ are canonically isomorphic.
\end{theorem}
\vskip1cm
\npa {\em locally trivial vector bundles}. Let $(U_f,\psi_f)_{f\in \Hom(M,N)}$ be an atlas for $\Hom(M,N)$. Smooth curves in $U_f\subset \Hom(M,N)$  are just the images under $\psi_f$ of smooth sections of the bundle $pr_{2}^{\ast}E(f)\to \mathbb{R}\times M$, $pr_{2}:\mathbb{R}\times M \to M$, which have compact support in $M$ locally in $\mathbb{R}$.

Let $(V,\pi_{V},M)$ and $(W,\pi_{W},N)$ be two smooth vector bundles of finite rank over $M$ and $N$ respectively.
Define $(\tilde{V},\pi, \Hom(M,N))$ as the mapping space that to each $f:M\to N$ associates the space of sections $C^{\infty}(M,V\otimes f^{\ast}W)$ restricted to the graph of $f$ with compact support (endowed with the $\mathfrak{D}^{\infty}$-topology).
\begin{proposition}\label{infinite-vb}$(\tilde{V},\pi, \Hom(M,N))$ is a smooth vector bundle over $\Hom(M,N)$.
\end{proposition}
\begin{proof}The proof consists in the introduction of a local trivialization for $(\tilde{V},\pi, \Hom(M,N))$ adapted to the atlas $(U_f,\psi_f)_{f\in \Hom(M,N)}$. Besides the torsion free connection on $N$, we consider $W$ endowed with a connection, $\nabla_W$. For each $f \in \Hom(M,N)$ consider $\eta \in C^{\infty}(E(f))$ a smooth section with compact support of the pullback bundle $f^{\ast}TN$, and define the associated curve in $\Hom(M,N)$,
$$\gamma_{t}:=\psi_f(t\cdot\eta),\;t\in [0,1]\subset U_f,$$
$$\gamma_0=f,\;\;\gamma_{1}=g,$$
with $t\cdot\eta$ a section of $pr_{2}^{\ast}E(f)\to [0,1]\times M \subset \mathbb{R}\times M$ obtained by fiberwise multiplication, i.e. $t\cdot\eta(x)=t(\eta(x))$.
The parallel transport $\mathrm{P\;exp}(\gamma_t,\nabla_W,t)$ along the curve $\gamma_t$ induces an isomorphism between the spaces of sections $C^{\infty}(M,f^\ast W)$ and $C^{\infty}(M,g^\ast W)$.
Taking into account that the trivialization for $V$ in $\tilde{V}$ is obvious, we immediately get a trivialization of $\tilde{V}$ over $U_f$.
\end{proof}
\begin{rem} The previous result is just a generalization of the following result in differential geometry; given $(V,\pi,M)$ a vector bundle over $M$ with connection $\nabla_V$, and $U\subset M$ an open ball in $M$: parallel transport along rays in $U$ gives a smooth trivialization of $V$ over $U$, by identifying the fibres with the parallel transport map.
\end{rem}

\vskip 3mm \npa Follow the proof of Proposition \ref{infinite vector bundle}. Let $(U_{f},\psi_{f})_{f\in \Hom(M,N)}$ be a smooth atlas for $\Hom(M,N)$ as before; for each $g\in U_{f}$ consider the corresponding mapping space (typical fiber) given by the space of sections of $\widetilde{V}$ as in Proposition \ref{bundle-identification} $$\tilde{V}_{g}=C^{\infty}(M,S_{+}^{(k)} ( W^*[1])\otimes g^* (TN)\oplus S^{(k)} ( W^*[1])\otimes g^*(V[1])).$$ Since the hypothesis of Proposition $\ref{infinite-vb}$ for the data $(\tilde{V}_{g\in U_{f}},(U_{f},\psi_{f})_{f\in \Hom(M,N)})$ are verified, the result follows. $\square$

\vskip 3mm \npa
\begin{proposition}\label{categorical-property}The categorical property
\beq
\Hom({\mathbf{Z}\times \mathbf{M},\mathbf{N}})=\Hom(\mathbf{Z},\uHom(\mathbf{M},\mathbf{N})),\;\;\forall \mathbf{Z},\mathbf{M},\mathbf{N} \in {\mathrm{\bf SMan}}
\eeq
is verified.
\end{proposition}
\begin{proof}Let's denote with $\Pi L, \Pi W,\Pi V$ three vector-bundles with bases respectively the smooth manifolds $Z,M,N$, diffeomorphic to $\mathbf{Z},\mathbf{M},\mathbf{N}$. For the right-hand side of the previous equality let's choose a map $F\in\Hom({Z},\Hom(M,N))$; such a map locally taking value in a generic open set $U_{f}\subset \Hom(M,N)$, consists of a section of the following bundle
\beq
    \xymatrix{  pr_{2}^{*}E(f)\ar[d]  \\
    Z \times M  \\
               }
  \eeq
  with $pr_{2}:Z \times M \to M$ the projection in the second element. By the use of Theorem $\ref{Mic2}$ the tangent bundle to $\Hom(M,N)$ is identified with the sections with compact support of the pullback tangent bundle to $N$. Applying the same arguments used in the proof of Proposition \ref{result1} we get a one to one correspondence between
  $\Hom(\mathbf{Z},\uHom(\mathbf{M},\mathbf{N}))$
and sections of
\beq\label{rhs}
\xymatrix
{{\{S_{+}^{(q)} (pr_{1}^* L^*[1])\otimes pr_{2}^* (E(f))\oplus S^{(q)} (pr_{1}^* L^*[1])\otimes pr_2^*(\tilde{V}_{f}[1])\}_{f\in \Hom}}\ar[d]\\
Z\times M\\
}
\eeq
with $q$ the rank of $\Pi L$.
For left-hand side of the equality, let's choose a map $G\in\Hom(Z\times M,N)$ which corresponds to $F$ for the categorical property $\Hom({{Z}\times M,N})=\Hom({Z},\Hom(M,N))$ in $\mathrm{\bf Man}$.
Due to Proposition $\ref{result1}$, the sections of
$\Hom({\mathbf{Z}\times \mathbf{M},\mathbf{N}})$ are in one to one correspondence with the sections of
\beq
\label{lhs}
\begin{array}{c}{S_{+}^{(q)}\!(pr_{1}^* L[1])\!\otimes\! S^{(k)}_{+}\!(pr_{2}^*W[1])\!\otimes\! G^*(TN)\!\oplus\! S^{(q)}\!(pr_{1}^* L[1])\!\otimes\! S^{(k)}\!(pr_{2}^*W[1])\!\otimes\! G^*({V}[1])} \\ \downarrow\\
 Z\times M .\\
\end{array}
\eeq
 The sections of $(\ref{lhs})$ and $(\ref{rhs})$ coincide by construction.
\vskip 0.5cm
{\bf Addendum}. Actually what is crucial in this identification is the fact that the space of sections of $\{pr_{2}^{*}E(f)\}_{f\in \Hom(M,N)}\to Z\times M$ for a given $F\in\Hom({Z},\Hom(M,N))$ coincides with the space of sections of $G^*(TN)\to Z\times M$, with $F,G$ related by the mentioned categorical property. $\square$

\end{proof}

\section {Appendix}

In the category ${\bf SMan}$ the finite product ($\times$) is defined in the following way:
given $\mathcal{M}=(M,\mathcal{O}_{M}), \mathcal{N}=(N,\mathcal{O}_{N})\in {\bf SMan}$, $\mathcal{M}\times  \mathcal{N}$ is the graded-manifold obtained attaching to each rectangular open set $U\times V\subset M\times N$ the super-algebra $\mathcal{O}_{M\times N}(U\times V) :=  \mathcal{O}_{M}(U)\widehat{\otimes} \mathcal{O}_{N}(V)$, where $\mathcal{O}_{M}(U)\widehat{\otimes}\mathcal{O}_{N}(V)$ is the completion\footnote{in the projective tensor topology} of $\mathcal{O}_{M_{0}}(U){\otimes} \mathcal{O}_{N_{0}}(V)$.
In particular ${\bf SMan}$ is a monoidal category, with tensor product equal to the coproduct of ringed spaces (i.e. $\otimes:=\times$) and the identity given by the terminal object $\mathbb{R}^{0|0}$ in ${\bf SMan}$, i.e. the locally ringed space with the singleton $\{\ast\}$ as topological space and $\mathbb{R}$ as sheaf.
In every closed monoidal category the internal $\Hom$, for any two objects $\mathcal{N},\mathcal{M}$, is equipped with the canonical morphism (evaluation)
$$ev_{\mathcal{N},\mathcal{M}}:\uHom(\mathcal{N},\mathcal{M})\otimes \mathcal{N}\rightarrow \mathcal{M},$$
which is universal (for every object $X$ and morphism $f: X\otimes \mathcal{N} \rightarrow \mathcal{M}$, there exist a unique morphism $h:X\rightarrow \uHom(\mathcal{N},\mathcal{M})$, such that $f=ev_{\mathcal{N},\mathcal{M}}\circ (h\otimes id_{\mathcal{N}})$). The introduction of the morphism $ev$ can be easily understood by the use of the definition of $\uHom$ applied to
\beq
\Hom(\uHom(\cN,\cM)\otimes \cN,\cM)=\Hom(\uHom(\cN,\cM),\uHom(\cN,\cM))
\eeq
and then considering on the right hand side the identity morphism $id_{\uHom(\cN,\cM)}$.
Other useful properties of internal $\mathrm{Hom}$-functor of a closed monoidal category are
the isomorphism $$\uHom(\mathcal{M} \otimes \mathcal{N},\mathcal{G})\simeq \uHom(\mathcal{M},\uHom(\mathcal{N},\mathcal{G}));$$
the composition of internal homomorphism $$\uHom(\mathcal{M},\mathcal{N})\otimes \uHom(\mathcal{N},\mathcal{K})\rightarrow \uHom(\mathcal{M},\mathcal{K}),$$
and the product of internal homomorphisms by identities $$\uHom(\mathcal{M},\mathcal{N})\rightarrow \uHom(\mathcal{K}\otimes \mathcal{M},\mathcal{K} \otimes \mathcal{N}).$$

Giuseppe Bonavolont\`a,\\
Campus Kirchberg, Mathematics Research Unit\\ 6, rue R. Coudenhove-Kalergi, L-1359 Luxembourg
City
\vskip 0.5cm
Alexei Kotov,\\
Institutt for matematikk og statistikk\\
Fakultet for naturvitenskap og teknologi
Universitetet i Troms{\o}
N-9037 Troms{\o}

\end{document}